\PassOptionsToPackage{dvipsnames}{xcolor}

\documentclass{article}
\usepackage[left=3cm, right=3cm, top=4cm, bottom=4cm]{geometry}
\usepackage{caption}
\usepackage{subcaption}
\usepackage{array}
\usepackage[hidelinks]{hyperref}
\usepackage{graphicx}
\usepackage[square,sort,comma,numbers]{natbib}
\usepackage[utf8]{inputenc}
\usepackage[english]{babel}
\usepackage{amsmath}
\usepackage{mathtools}
\usepackage{amssymb}
\usepackage{textcomp}
\usepackage{blindtext}
\usepackage{pifont}
\usepackage{amsthm}
\usepackage{dsfont}
\usepackage[shortlabels]{enumitem}
\usepackage{float}
\usepackage{tikz-cd}
\usepackage{amsfonts}
\usepackage{mathabx}
\usepackage{xcolor}

\newtheorem{theorem}{Theorem}[section]

\newtheorem{prop}[theorem]{Proposition}

\theoremstyle{definition}
\newtheorem{example}[theorem]{Example}
\newtheorem{remark}[theorem]{Remark}

\newtheorem{question}[theorem]{Question}

\definecolor{palestinianGreen}{RGB}{0,128,0}
\definecolor{palestinianRed}{RGB}{206,17,38}

\AtEndDocument{\bigskip{\footnotesize
  
   \textsc{Alejandro Vicente, The Hebrew University of Jerusalem, Jerusalem, Israel.} \par  
  \textit{E-mail}:  \texttt{ kvicente931207@gmail.com}
  }}

\title{On Lagrangian barriers and applications of Non-Squeezing}
\author{Alejandro Vicente}
\date{}

\begin{document}

\maketitle

\begin{abstract}
    In this note we show the existence of Lagrangian barriers in a certain class of domains in $\mathbb{R}^{2n}$, including dual Lagrangian products and some ``sufficiently" round domains. Many of these results come as applications of the Non-Squeezing Theorem. We also give a new interesting application of the Non-Squeezing Theorem and the Symplectic Camel Theorem. 
\end{abstract}

\section{Introduction}\label{sec: intro}

Symplectic embedding problems have guided the development of Symplectic Topology since the mid-1980s. See \cite{Schlenk2017SymplecticEP}, for a thorough account of the history and advances on this topic. Many different strange phenomena arising in symplectic embedding problems have put this subject right at the core of the area. These include the staircases first discovered in \cite{Mcduff2009TheEC}, the Lagrangian barriers uncovered in \cite{Biran2001LagrangianBA}, and the more recent symplectic barriers found in \cite{HaimKislev2023OnTE}. In this paper we present new examples of symplectic manifolds where the Lagrangian barrier phenomenon is present.

Let 
\begin{equation*}
    \begin{split}
        B^{2n}(r)&:=\{(x_1,\ldots,x_n,y_1,\ldots,y_n)\in \mathbb{R}^{2n}\mid x_1^2+\ldots x_n^2+y_1^2+\ldots y_n^2\leq r/\pi\},\\
        Z^{2n}(r)&:=B^2(r)\times \mathbb{R}^{2n-2}.
    \end{split}
    \end{equation*}

The first examples of Lagrangian barriers were discovered by Biran in \cite{Biran2001LagrangianBA}. Perhaps the most famous example of this phenomenon is the result in \cite{Biran2001LagrangianBA} saying that for any symplectic embedding $\phi:B^{2n}(\lambda)\mapsto\mathbb{CP}^n$ and $\lambda\geq 1/2$, the image of $\phi$ intersects the Lagrangian submanifold $\mathbb{RP}^n\subset \mathbb{CP}^n$. From the ideas of the original proof of this result, a similar statement can be derived: for any symplectic embedding $\phi:B^{2n}(\lambda)\mapsto B^{2n}(1)$ and $\lambda\geq 1/2$, the image of $\phi$ intersects the Lagrangian disk $\mathcal{L}:=\{(\mathbf{x},\mathbf{y})\in T^*\mathbb{R}^n\mid \mathbf{y}=0, ||\mathbf{x}||\leq 1\}$. Another proof of the latter result was given recently in \cite{Sackel2021OnCQ}. This result hints at the possibility of the existence of a wealth of regions in $\mathbb{R}^{2n}$ admitting Lagrangian barriers. In this paper, we present some results in this direction, but first, we set up some notation.

A \textit{symplectic capacity} $c$ is a map that assigns a non-negative real number or infinity, to a symplectic manifold $(M,\omega)$ such that the following conditions are met:
\begin{itemize}
\item (Monotonicity) If $(M_1,\omega_1)$ symplectically embeds into $(M_2,\omega_2)$, then $c(M_1,\omega_1)\leq c(M_2,\omega_2)$,
\item (Conformality) For $\lambda\neq 0$, we have that $c(M,\lambda \omega)=|\lambda|c(M,\omega)$,
\item (Non-triviality) $c(B^{2n}(1),\omega_0)>0$ and $c(Z^{2n}(1),\omega_0)<\infty$.
\end{itemize}
In addition, we say that a symplectic capacity is \textit{normalized} if $c(B^{2n}(1),\omega_0)=c(Z^{2n}(1),\omega_0)=1$.
    
The most classical example of a (normalized) symplectic capacity is the Gromov width. Given a $2n$-dimensional symplectic manifold $(X,\omega)$, its Gromov width $c_{Gr}(X,\omega)$ is defined to be the supremum for all $r > 0$ such that there exists a symplectic embedding $(B^{2n}(r),\omega_0) \rightarrow (X,\omega)$, where $\omega_0$ is the standard symplectic form in $\mathbb{R}^{2n}$. Notice that Biran's above mentioned results can now be stated as $c_{Gr}(\mathbb{CP}^n\setminus \mathbb{RP}^n)=1/2$ and $c_{Gr}(B^{2n}\setminus \mathcal{L})=1/2$. 

The motivating question for this work is the following: 
\begin{question}\label{quest: capacities}
Let $X\subset \mathbb{R}^{2n}$ be a compact domain and let $\mathcal{L}_X:=X\cap \{(\mathbf{x},\mathbf{y})\in T^*\mathbb{R}^n\mid \mathbf{x}=0\}$. Let $c$ be any normalized symplectic capacity, is it true that 
$$c(X\setminus\mathcal{L}_X)<c(X),$$
holds? In the case where this is true, can we compute $c(X\setminus\mathcal{L}_X)$?
\end{question} 
We will see in Theorems \ref{thm: puncture_higherdim}, \ref{thm: punctured_cube} and \ref{thm: estimative_2} that in many cases the answer to the first part of this question is true. Notice that we are just removing (the intersection of a domain with) a Lagrangian plane and this already has the effect of decreasing symplectic capacities significantly so, motivated by the results in \cite{HaimKislev2023OnTE} we could consider a version of Question \ref{quest: capacities} reloaded.

\begin{question}
Let $X\subset \mathbb{R}^{2n}$ be a compact domain. For $\mathbf{x}_0\in \mathbb{R}^n$ fixed, let $\mathcal{L}_X^{\mathbf{x}_0}:=X\cap \{(\mathbf{x},\mathbf{y})\in T^*\mathbb{R}^n\mid \mathbf{x}=\mathbf{x}_0\}$. Let $c$ be any normalized symplectic capacity and $\varepsilon>0$, does there exist points $\mathbf{x}_0,\ldots,\mathbf{x}_k\in \mathbb{R}^n$ such that 
$$c\left(X\setminus(\mathcal{L}_X^{\mathbf{x}_0}\cup\ldots\cup\mathcal{L}_X^{\mathbf{x}_k})\right)<\varepsilon,$$
holds?
\end{question}

For a symplectic manifold $(X,\omega)$ with contact boundary, denote by $\mathcal{A}_{min}(\partial X)$ the minimal action of Reeb orbits on $\partial X$. It is a classical result that for convex domains in $\mathbb{R}^{2n}$, the Hofer-Zehnder capacity agrees with the minimal action. In this note, we also analyze the effect that removing Lagrangian planes has on the minimal action. 

We will see some cases where all normalized symplectic capacities in $X\setminus \mathcal{L}$ agree with the minimal action of Reeb orbits in $\partial(X\setminus \mathcal{L}_X)$, i.e. $c\left(X\setminus\mathcal{L}_X\right)=\mathcal{A}_{min}(\partial (X\setminus\mathcal{L}_X))$. In particular, it is interesting to notice that when removing more than one Lagrangian plane from a region $X\subset \mathbb{R}^{2n}$, it is not necessarily true anymore that any normalized symplectic capacity $c$ agrees with the minimal action, i.e. the equality
$$c\left(X\setminus(\mathcal{L}_X^{\mathbf{x}_0}\cup\ldots\cup\mathcal{L}_X^{\mathbf{x}_k})\right)=\mathcal{A}_{min}(\partial (X\setminus(\mathcal{L}_X^{\mathbf{x}_0}\cup\ldots\cup\mathcal{L}_X^{\mathbf{x}_k}))),$$
does not necessarily holds.\\
To see this, take the points $\mathbf{x}_0=((k-1)/k,0)$ and $\mathbf{x}_1=((1-k)/k,0)$, in $\mathbb{R}^2$,for $k$ big enough and let $X:=D^2(1)\times_L D^2(1)$. Then we can easily see that $\mathcal{A}_{min}(\partial (X\setminus(\mathcal{L}_{X}^{\mathbf{x}_0}\cup\mathcal{L}_{X}^{\mathbf{x}_1})))=2/k$, while $c(X\setminus(\mathcal{L}_{X}^{\mathbf{x}_0}\cup\mathcal{L}_{X}^{\mathbf{x}_1}))\geq\sqrt{(k-1)/k}$. The minimal action in this case is realized by the 2-bouncing billiard trajectory between the points $\mathbf{x}_0$ and $(1,0)$.

Before we present our main theorems, we recall some  results on the symplectic measurements of two of our main objects: the Lagrangian bidisk and the Lagrangian product of a cube in $\mathbb{R}^n$ and its dual polar body.

Let 
\begin{equation*}
    \begin{split}
        D^n(r)&:=\{(x_1,\ldots,x_n)\in \mathbb{R}^n\mid x_1^2+\ldots x_n^2\leq r^2\},\\
        \square^n(a_1,\ldots,a_n)&:=\{(x_1,\ldots,x_n)\in \mathbb{R}^n\mid |x_i|<a_1\},\\
        \diamond^n(r)&: =\{(y_1,\ldots,y_n)\in \mathbb{R}^n\mid |y_1|+\ldots+|y_n|<r\}.   
    \end{split}
    \end{equation*} 
    
For $A\subset \mathbb{R}^n_{(x_1,\ldots,x_n)}$ and $B\subset \mathbb{R}^n_{(y_1,\ldots,y_n)}$, we define the \textit{Lagrangian product} $A\times_L B$ as the usual cartesian product, where $A\times_L B$ is regarded as a subset of $\mathbb{R}^{2n}\cong \mathbb{R}^n\times_L\mathbb{R}^n$, with the symplectic form $\sum_{i=1}^{n}dy_i\wedge dx_i$. In particular, notice that for any $A\subset \mathbb{R}^n$, we have that $A\times_L D^n(1)$ is the disk cotangent bundle of $A$.

\begin{example}\label{ex: D^2xD^2}
    Let $A=D^n(1)$ and $B=:A^{\circ}=D^n(1)$ its dual body. It follows from an idea of Y. Ostrover, outlined in \cite{Schlenk2017SymplecticEP}, that for any normalized symplectic capacity $c$, we have that $c(A\times_L B)=4$. 
\end{example}

\begin{example}\label{ex: cube_x_diamond}
    Let $A:=\square^n(1,\ldots,1)$ be a cube and $B:=A^{\circ}=\diamond^n(1)$, its dual body . We have that $A\times_L B$ is symplectomorphic to $B^{2n}(4)$, by the main result in \cite{Ramos2017OnTR} and so $c(A\times_L B)=4$.
\end{example}

%It is an open problem to confirm wether all normalized symplectic capacities coincide in the class of \textit{dual Lagrangian products}, i.e. Lagrangian products of a region $K\subset \mathbb{R}^n$ and its dual polar body. This is the strong Viterbo conjecture restricted to the class of dual Lagrangian products, see \citep{Gutt2022ExamplesAT} for an account on this problem. Examples \ref{ex: D^2xD^2} and \ref{ex: cube_x_diamond} confirm that the answer is positive for $K$ being the unit ball in the 2-norm and the $\infty$-norm in $\mathbb{R}^n$. In this note, we prove the strong Viterbo conjecture  for the case where $K\subset \mathbb{R}^n$ is the unit ball $D^{n}_p(1)$ in the $p$-norm. Notice in this case, the polar dual body $K^{\circ}$ of $K$ is the unit ball $D^{n}_q(1)$ for $q$ such that $1/p+1/q=1$. Our proof follows an idea taken from \cite{Latschev2011TheGW}.
%
%
%
%\begin{remark}
%    Although the Viterbo Conjecture has been proven to be false in the general setting in which it was initially proposed, see \cite{HaimKislev2024ACT} for the first counterexample, there is still hope that its statement would hold in the class dual Lagrangian products $K\times K^{\circ}$, for $K\subset \mathbb{R}^n$ a centrally symmetric convex body. Theorem \ref{thm: viterbo} provides modest evidence for such hope.
%\end{remark}
%
%
%\begin{theorem}\label{thm: viterbo}
%For any normalized symplectic capacity $c$ and any $p>0$,
%$$c(D_p^{n}(1)\times_L D_q^{n}(1))=4.$$
%\end{theorem}

We now state our two main results.

\begin{theorem}\label{thm: puncture_higherdim}
Let $c$ be any normalized symplectic capacity, then the disk cotangent bundle $D^n(1)\setminus D^n(\delta)\times_L D^n(1)$ of the annulus $D^n(1)\setminus D^n(\delta)$ satisfies
$$c(D^n(1)\setminus D^n(\delta)\times_L D^n(1))=2(1-\delta).$$
Furthermore, $\mathcal{A}_{min}(D^n(1)\setminus D^n(\delta)\times_L D^n(1))=2(1-\delta)$.
\end{theorem}

\begin{theorem}\label{thm: punctured_cube}
Let $c$ be any normalized symplectic capacity, then the Lagrangian product of $ \diamond^n(1)$  and $\square^n(1,\dots,1)\setminus\{(0,\ldots,0)\}$ satisfies
$$c(\square^n(1,\dots,1)\setminus\{(0,\ldots,0)\}\times_L \diamond^n(1))=2.$$
Furthermore, $\mathcal{A}_{min}(\square^n(1,\dots,1)\setminus\{(0,\ldots,0)\}\times_L \diamond^n(1))=2$.
\end{theorem}

The lower bounds in Theorems \ref{thm: puncture_higherdim} and \ref{thm: punctured_cube} come from Theorem \ref{thm: Gr_holed_K} to be proven in Section \ref{sec: lag barr}, while the upper bounds come as consequence of the Non-Squeezing Theorem for the former and an embedding obstruction for the latter, we will see these proofs in Section \ref{sec: lag barr}.

Let $K\subset \mathbb{R}^n$, we say that $K$ is $\alpha$-pinched, for $\alpha>1$, if $\widehat{k}\leq \alpha\widecheck{k}$, where $\widehat{k}:=\max_{x\in \partial K}||x||$ and $\widecheck{k}:=\min_{x\in \partial K}||x||$. The next result, consequence of Example \ref{ex: D^2xD^2} and Theorem \ref{thm: puncture_higherdim}, shows that Lagrangian barriers are common, they exist in many dual Lagrangian products.

\begin{theorem}\label{thm: estimatives}
Let $c$ be any normalized symplectic capacity and let $K\subset \mathbb{R}^n$ be a convex domain containing the origin. The following inequalities hold:
\begin{enumerate}
\item  \label{item: 1_estimatives} $4\sqrt{\frac{\widecheck{k}}{\widehat{k}}}\leq c(K\times_L K^{\circ})\leq 4\sqrt{\frac{\widehat{k}}{\widecheck{k}}},$
\item \label{item: 2_estimatives} $2\sqrt{\frac{\widecheck{k}}{\widehat{k}}}\leq c(K\setminus \{(0,0)\}\times_L K^{\circ})\leq 2\sqrt{\frac{\widehat{k}}{\widecheck{k}}}.$
%\item \label{item: 3_estimatives} $\inf_{K\in \mathcal{S}} \frac{c_{Gr}(K\setminus \{(0,0)\}\times_L K^{\circ})}{c_{Gr}(K\times_L K^{\circ})}\geq \frac{\widecheck{k}}{2\widehat{k}}.$
\end{enumerate}
In particular, for $K\subset{\mathbb{R}^2}$ be an $\alpha$-pinched symmetrically convex domain, for $\alpha=\left(\frac{4}{\pi}\right)^{2/3}\approx 1.1751$, we have that
\begin{equation}\label{eq: inequality_est}
c(K\setminus \{(0,0)\}\times_L K^{\circ})< \sqrt{\textup{Vol}(K)\textup{Vol}(K^{\circ})}\leq c(K\times_L K^{\circ}),
\end{equation}
i.e. $\{(0,0)\}\times_LK^{\circ}$ is a Lagrangian barrier inside $K\times_L K^{\circ}$.

\end{theorem}

Similarly to Theorem \ref{thm: estimatives}, we can use Biran's result, mentioned above, to get Lagrangian barriers in more general settings, as the next theorem shows.

\begin{theorem}\label{thm: estimative_2}
Assume $M\subset \mathbb{R}^{2n}$ is a convex region containing the origin. Let 
$$\widehat{m}:=\max_{x\in \partial M} ||x||,$$
and 
$$\widecheck{m}:=\min_{x\in \partial M} ||x||.$$
Assume also $M$ is $\sqrt{2}$-strictly pinched, i.e. $\widehat{m}<\sqrt{2}\widecheck{m}$. Then
$$\frac{1}{4}c(M)\leq c(M\setminus \mathcal{L})\leq c(M).$$  
In particular, if $c(M)> \pi\widecheck{m}^2$, then the right-most inequality is strict.
\end{theorem}
\begin{proof}
Clearly $B^{2n}(\pi \widecheck{m}^2)\setminus \mathcal{L}\subseteq M\setminus \mathcal{L}\subseteq B^{2n}(\pi \widehat{m}^2)\setminus \mathcal{L}$, which implies by Biran's result that 
$$\frac{\pi\widecheck{m}^2}{2}\leq c(M\setminus \mathcal{L})\leq \frac{\pi\widehat{m}^2}{2}$$
But we also know that 
$$\pi\widecheck{m}^2\leq c(M)\leq \pi\widehat{m}^2.$$
So, using the pinching condition, the claim follows. 
\end{proof}

%\begin{example}
%    $A=S^2$ and $B=S^2$. Using the standard toric action in $\mathbb{CP}^1\times \mathbb{CP}^1$ and a result in \cite{CristofaroGardiner2020OnIS}, we can see that $c_{Gr}(A\times B)=4\pi$ whereas $c_{Gr}(A\times_L B)=2\pi$, by Theorem \ref{thm: c_Gr S2xS2}.
%\end{example}

%\begin{example}(Private communication with Yaron Ostrover)
    %$A=B^2(1)\setminus\{(0,0)\}$ and $B=B^2(1)$. Using the standard toric action in $A\times B$ and the Traynor trick we can see that $c_{Gr}(A\times B)\geq1$. Since $A\times B\subset B^2(1)\times B^2(1)$ we have that $c_{Gr}(A\times B)\leq 1$. Therefore, $c_{Gr}(A\times B)=1$. On the other hand, $A\times_L B\subset B^4(2)\setminus \{(0,0)\times \mathbb{R}^2\}$. A famous result by Biran \cite{Biran2001LagrangianBA} says that the intersection of the plane $\{(0,0)\times \mathbb{R}^2\}$ with the ball $B^4(2)$ is a Lagrangian barrier, meaning that $c_{Gr}(A\times_L B)\leq c_{Gr}( B^4(2)\setminus \{(0,0)\times \mathbb{R}^2)=1$. Using techniques from integrable systems it can be shown that $c_{Gr}(A\times_L B)\geq \frac{2}{\sqrt{\pi}}$ \textcolor{red}{Check this!}.
%\end{example}

\subsection{Other applications of Non-Squeezing}

As commented above, the proofs of Theorems \ref{thm: puncture_higherdim} and \ref{thm: punctured_cube} rely on the celebrated Gromov's Non-Squeezing Theorem. In this section we give two more interesting applications of Non-Squeezing to symplectic embedding problems, whose proof follow the same ideas as those in the proofs of theorem \ref{thm: puncture_higherdim}. We begin by setting up some notation.

Let $a>0$ and $\mathbf{q}:=(q_1,q_2,q_3)\in \mathbb{R}^3$. We define the truncated cylinder:

$$C_a(R):=\left\{\mathbf{q}\in \mathbb{R}^3\mid q_1^2+q_2^2=\frac{R}{\pi}, |q_3|< a\right\}.$$
For the case $R=\pi$, we abbreviate $C_a:=C_a(\pi)$. Hence, the disk cotangent bundle of $C_a$, with the metric induced by its inclusion in $T^*\mathbb{R}^3\cong \mathbb{R}^6$, is given by:
$$D^*C_a=\{(\mathbf{q},\mathbf{p})\in \mathbb{R}^6\mid \mathbf{q}\in C_a,\langle \mathbf{q},\mathbf{p}\rangle=0, |\mathbf{p}|^2< 1\},$$
where $\mathbf{p}:=(p_1,p_2,p_3)\in \mathbb{R}^3$.

In order to state and prove the applications of Non-Squeezing, we first state the following necessary result, to be proven in Section \ref{sec: disk_cot}.

\begin{theorem}\label{thm: Gromov_width}
    For any normalized symplectic capacity $c$, we have that
\begin{equation}
   c(D^*C_a)= 
\begin{cases}
    \begin{split}
       4a&, \textup{ if } a\leq \pi\\
        4\pi&, \textup{ if } a\geq \pi.
    \end{split}
    \end{cases}
\end{equation}   
\end{theorem}
We also consider the unbounded cylinder $C(R)\subset \mathbb{R}^3$, to be defined as:
$$C(R):=\left\{\mathbf{q}\in \mathbb{R}^3\mid q_1^2+q_2^2=\frac{R}{\pi}\right\},$$
i.e. the limit of the spaces $C_a(R)$ for $a\rightarrow \infty$.

\begin{remark}
    Heuristically, the fact that $c_{Gr}(D^*C_a)=4\pi$ for $a$ big enough is not surprising, given that $D^*C_a$ converges to $D^*C(\pi)$ and in turn, $D^*C(\pi)$ can be thought of, roughly speaking, as the limit of the disk cotangent bundles $D^*\mathcal{E}(1,1,c)$ of the ellipsoids of revolution $\mathcal{E}(1,1,c)$, i.e.
    $$\mathcal{E}(1,1,c):=\left\{(x,y,z)\in \mathbb{R}^3\mid x^2+y^2+\frac{z^2}{c^2}=1\right\}$$
when $c\rightarrow \infty$, and it was shown in \cite{ferreira2023gromovwidthdiskcotangent}, by Ferreira, Ramos and the author, that $c_{Gr}(D^*\mathcal{E}(1,1,c))=4\pi$, for sufficiently large $c$.
\end{remark}

We are now ready to state and prove the first of the two interesting results of this section, which in some sense, can be thought itself as a Non-Squeezing phenomenon for symplectic embeddings of balls into the disk cotangent bundle of the unbounded cylinder.

\begin{theorem}\label{thm: NonSqueezing}
    The ball $B^4(r)$ symplectically embeds into $D^*C(R)$, if and only if, $r\leq 4R$.
\end{theorem}
\begin{proof}
    On one hand, it follows from Theorem \ref{thm: Gromov_width} that  the ball $B^4(4R)$ embeds into $D^*C_a(R)$ and therefore, also does the ball $B^4(r)$, for all $r\leq 4R$. On the other hand, let $(z,\theta)$ denote $\mathbf{q}$ in cylindrical coordinates and let $(p_z,p_{\theta})$ be the induced coordinates for $\mathbf{p}$. That is
    \begin{equation}\label{eq: cyl_coords}
    \begin{split}
        (q_1,q_2,q_3)&=\left(\sqrt{\frac{R}{\pi}}\cos{\theta}, \sqrt{\frac{R}{\pi}}\sin{\theta}, z\right),\\
        (p_1,p_2,p_3)&=\left(-\sqrt{\frac{\pi}{R}}p_{\theta}\sin{\theta} ,\sqrt{\frac{\pi}{R}}p_{\theta}\cos{\theta}, p_z\right).
    \end{split}
    \end{equation}
    Consider the map $f: D^*C(R)\rightarrow B^2\left(4R\right)\times \mathbb{R}^2$ given by
    $$f(z,\theta,p_z,p_{\theta})=\left(\sqrt{\frac{R}{\pi}}+p_{\theta},-\frac{\theta}{\sqrt{\frac{R}{\pi}}+p_{\theta}},z,p_z\right).$$
    It is not hard to see this is a symplectic embedding. For this notice that, the symplectic form in $D^*C_a$ in the cylindrical coordinates in \eqref{eq: cyl_coords} is given by $dp_{\theta}\wedge d\theta+dp_z\wedge dz$, whereas the symplectic form in the target space $B^2\left(4R\right)\times \mathbb{R}^2$, using radial coordinates $(r,\theta)$ in the first factor and cartesian coordinates $(x,y)$ in the second factor, is $-rdr\wedge d\theta+dy\wedge dx$. Hence, by monotonicity of the Gromov width and Gromov's Non-squeezing theorem, we have that $c_{Gr}(D^*C(R))\leq c_{Gr}(B^2(4R)\times \mathbb{R}^2)=4R$. So, $B^4(r)$ symplectically embeds into $D^*C(R)$, if and only if, $r\leq 4R$.
\end{proof}

Let $X_a^+:=\{\mathbf{q}\in \mathbb{R}^3\mid q_3=a, q_1^2+q_2^2\geq 1\}$ and $X_a^-:=\{\mathbf{q}\in \mathbb{R}^3\mid q_3=-a, q_1^2+q_2^2\geq 1\}$. We define the \textit{open dumbbell} $X_a$ to be the union $X_a^+\cup C_a \cup X_a^-$, see Figure \ref{fig: X_a}.  We denote by $\text{Emb}(B^4(r),D^*X_a)$ the space of all possible symplectic embeddings of the ball $B^4(r)$ into the disk cotangent bundle $D^*X_a$ of the open dumbbell $X_a$. Notice that, we can always find a copy of any size of the Lagrangian bidisk inside both $D^*X_a^{\pm}$, and therefore because of \cite{Ramos2015SymplecticEA}, we can find balls of any size inside both $D^*X_a^{\pm}$. Hence, the space $\text{Emb}(B^4(r),D^*X_a)$ is non-empty for every $r>0$. The next result is a ``symplectic camel"-like statement for the space $D^*X_a$. 

\begin{figure}[htbp] % h: here, t: top, b: bottom, p: page of floats
    \centering % Center the figure
    \begin{tikzpicture}[scale=1.5]

    % Define parameters for the parallelograms
    \def\planeWidth{3.4} % Width of the parallelogram
    \def\planeHeight{1.0} % Height of the parallelogram
    \def\tilt{1.0} % Tilt for the parallelograms

    % Upper plane (parallelogram) - flipped orientation
    \draw[thick] 
        (-\planeWidth/2+\tilt,1+\planeHeight/2) -- 
        (\planeWidth/2+\tilt,1+\planeHeight/2) -- 
        (\planeWidth/2-\tilt,1-\planeHeight/2) -- 
        (-\planeWidth/2-\tilt,1-\planeHeight/2) -- cycle;

    % Lower plane (parallelogram) - flipped orientation
    % Split into dashed and solid parts
    % Left solid part
    \draw[thick] 
        (-\planeWidth/2+\tilt,-1+\planeHeight/2) -- 
        (-0.5,-1+\planeHeight/2);
    % Dashed middle part
    \draw[dashed] 
        (-0.5,-1+\planeHeight/2) -- 
        (0.5,-1+\planeHeight/2);
    % Right solid part
    \draw[thick] 
        (0.5,-1+\planeHeight/2) -- 
        (\planeWidth/2+\tilt,-1+\planeHeight/2);

    % Draw the remaining three sides of the lower parallelogram as solid
    \draw[thick] 
        (\planeWidth/2+\tilt,-1+\planeHeight/2) -- 
        (\planeWidth/2-\tilt,-1-\planeHeight/2) -- 
        (-\planeWidth/2-\tilt,-1-\planeHeight/2) -- 
        (-\planeWidth/2+\tilt,-1+\planeHeight/2);

    % Cylindrical edges
    % Right vertical line
    \draw[dashed] (0.5,1) -- (0.5,{1-\planeHeight/2}); % Dashed part inside the top parallelogram
    \draw[thick] (0.5,{1-\planeHeight/2}) -- (0.5,-1); % Solid part outside

    % Left vertical line
    \draw[dashed] (-0.5,1) -- (-0.5,{1-\planeHeight/2}); % Dashed part inside the top parallelogram
    \draw[thick] (-0.5,{1-\planeHeight/2}) -- (-0.5,-1); % Solid part outside

    % Dashed ellipses
    \draw[dashed] (0,-1) ellipse [x radius=0.5, y radius=0.2]; % Bottom ellipse
    \draw[dashed] (0,1) ellipse [x radius=0.5, y radius=0.2]; % Top ellipse

    % Hidden curve of bottom ellipse
    \draw[dashed] (-0.5,-1) arc[start angle=180,end angle=0,x radius=0.5,y radius=0.2];

    % Add labels
    \node at (\planeWidth/2+0.6,1) {$q_3 = a$}; % Upper q_3
    \node at (\planeWidth/2+0.6,-1) {$q_3 = -a$}; % Lower q_3

    \node at (-\planeWidth/2-0.5,1) {$X_a^+$}; % X_a^+ label
    \node at (-\planeWidth/2-0.5,-1) {$X_a^-$}; % X_a^- label

    \node at (0.7,0) {$C_a$}; % C_a label

    \end{tikzpicture}
    \caption{$X_a$.} % Add caption
    \label{fig: X_a} % Add label for referencing
\end{figure}

\begin{theorem}\label{thm: camel}
    For any $a>0$, the space $\text{Emb}(B^4(r),D^*X_a)$ is disconnected for any $r> 4\pi$.
\end{theorem}
\begin{proof}
    In cylindrical coordinates 
    \begin{equation}\label{eq: cyl_coords2}
    \begin{split}
        (q_1,q_2,q_3)&=\left(r\cos{\theta}, r\sin{\theta}, z\right),\\
        (p_1,p_2,p_3)&=\left(p_r\cos\theta-\frac{p_{\theta}}{r}\sin{\theta} ,p_r\sin\theta +\frac{p_{\theta}}{r}\cos{\theta}, p_z\right),
    \end{split}
    \end{equation}
    in $T^*\mathbb{R}^3$, we have that
    $$D^*X_a^{\pm}=\left\{(r,\theta,z,p_r,p_{\theta},p_z)\in T^*\mathbb{R}^3\mid z=\pm a, r\geq 1,p_z=0, p_r^2+\frac{p_{\theta}^2}{r^2}< 1\right\},$$
    and as before,
$$D^*C_a=\left\{(r,\theta,z,p_r,p_{\theta},p_z)\in \mathbb{R}^3\mid r=1,|z|\leq a, p_r=0,p_z^2+p_{\theta}^2<1\right\}.$$
Now, consider the symplectic manifold $(\mathbb{R}^4, \omega_0)$ where $\mathbb{R}^4=\mathbb{R}^2\times \mathbb{R}^2$, with polar coordinates $(r,\theta)$ in the first factor and Cartesian coordinates $(x,y)$ in the second factor. Here the symplectic form is given by $\omega_0=-rdr\wedge d\theta +dy\wedge dx$. The symplectic form in $D^*X_a$, with coordinates induced by inclusion in $T^*\mathbb{R}^3$ as in \eqref{eq: cyl_coords2}, is given by $dp_r\wedge dr+dp_{\theta}\wedge d\theta+dp_z\wedge dz$. Let $M\subset \mathbb{R}^4$ be the submanifold defined as the complement of the set
\begin{equation*}
    \tilde{M}:=\{(r,\theta,x,y)\in \mathbb{R}^4\mid y=0, r^2+x^2>4\},
\end{equation*}
i.e. $M=\mathbb{R}^2\setminus \tilde{M}.$ The famous symplectic camel theorem states that the  space $\text{Emb}(B^4(R),M)$ is disconnected for every $R>4\pi$, see for example \cite{Gromov}, for a proof.

We are now going to define a symplectic embedding $g: D^*X_a\rightarrow M$, satisfying the property that, the $y$-coordinate of the image by $g$ of any point in $D^*X_a^+$ or $D^*X_a^-$ is, strictly positive or negative, respectively.

\begin{equation*}
g(r,\theta,z,p_r,p_{\theta},p_z)=
\begin{cases}
    \begin{split}
\left(1+p_{\theta},-\frac{\theta}{1+p_{\theta}},z,p_z\right)&, \textup{ if } (r,\theta,z,p_r,p_{\theta},p_z)\in D^*C_a,\\
\left(1+p_{\theta},-\frac{\theta}{1+p_{\theta}}, r-1+a,p_r\right)&, \textup{ if } (r,\theta,z,p_r,p_{\theta},p_z)\in D^*X_a^{+},\\
\left(1+p_{\theta},-\frac{\theta}{1+p_{\theta}}, -r-1-a,-p_r\right)&, \textup{ if } (r,\theta,z,p_r,p_{\theta},p_z)\in D^*X_a^{-}.
    \end{split}
\end{cases}
\end{equation*}
The symplectic embedding $g$ induces a continuous map $\text{Emb}(B^4(r),D^*X_a)\hookrightarrow \text{Emb}(B^4(r),M)$, for every $r>0$, whose image intersects at least two different connected components of $\text{Emb}(B^4(r),M)$. To see this, just notice that we can embed the $B^4(r)$, for any $r>0$, into both $D^*X_a^{\pm}$ and the image of these balls, by the embedding $g$, falls into different sides of $M$ with respect to the disk $\{y=0, r^2+x^2\leq 4\}\subset M$.
Therefore, $\text{Emb}(B^4(r),D^*X_a)$ must be disconnected for $r>4\pi$.
\end{proof}

\noindent{\bf Structure of the paper:} In Section \ref{sec: lag barr}, we give the proofs of Theorems  \ref{thm: punctured_cube}, \ref{thm: estimatives} and \ref{thm: estimative_2}. In Section \ref{sec: disk_cot}, we give the proof of Theorem \ref{thm: Gromov_width}. 

\noindent{\bf Acknowledgements:} The author thanks Y. Ostrover for conversations that lead to work on these results and Vinicius G. B. Ramos for valuable comments and suggestions on an earlier version of this manuscript. A. V. was supported by the ISF Grant No. 2445/20.

\section{Lagrangian barriers}\label{sec: lag barr}

We begin this section by proving the following lower bound on the symplectic capacity of the Lagrangian product of an annular region (possibly a punctured region) and the dual body of the convex hull of this annular region.

\begin{theorem}\label{thm: Gr_holed_K}
Let $0\leq\delta<1$ and $K\subset \mathbb{R}^n$ be a convex and centrally symmetric domain. Then, for any symplectic capacity, not necessarily normalized, we have that 
$$c(K\setminus \delta K \times_L K^{\circ})\geq \frac{1-\delta}{2}c(K\times_L K^{\circ}).$$
For $\delta=0$, we think of $K\setminus \delta K$ as $K\setminus \{(0,0)\}$.
\end{theorem}

\begin{figure}[htbp]
    \centering
    \begin{tikzpicture}

        % Define custom color
        \definecolor{palestinianRed}{rgb}{0.9, 0.1, 0.1}

        % Draw axes
        \draw[->] (-3,0) -- (3,0) node[right] {};
        \draw[->] (0,-3) -- (0,3) node[above] {};

        % Helper for rotation: direction (1, 2)
        \def\rotationAngle{63.4349} % arctan(2/1) in degrees

        % Fill the space between the two ellipses
        \begin{scope}
            % Fill the region between the ellipses
            \fill[RoyalBlue!70, opacity=0.8, rotate around={\rotationAngle:(0,0)}] 
                (0,0) ellipse (2.0cm and 1.2cm)
                [even odd rule]
                (0,0) ellipse (0.5cm and 0.2cm);
        \end{scope}

        % Draw larger ellipse (black outline)
        \draw[thick, rotate around={\rotationAngle:(0,0)}] (0,0) ellipse (2.0cm and 1.2cm);

        % Draw smaller ellipse (black outline)
        \draw[thick, rotate around={\rotationAngle:(0,0)}] (0,0) ellipse (0.5cm and 0.2cm);

        % Add label to line K
        \node[below left] at (-1,2) {$K$};
        \node[below left] at (-0.2,-0.6) {\textcolor{black}{$K\setminus\delta K$}};

    \end{tikzpicture}
    \caption{The set $K\setminus \delta K$.}
\end{figure}
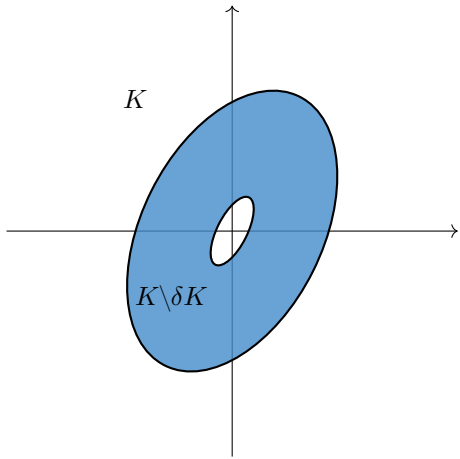

\begin{proof}
Let $h_K:S^{n-1}\rightarrow \mathbb{R}$ be the function defining the boundary $\partial K$ of $K$, i.e. if $z\in \partial K$ then $z=h_k\left(z/||z||\right)z/||z||$. We claim that the affine transformation
\begin{equation*}
\begin{split}
f_{\delta}:\frac{1-\delta}{2}K &\rightarrow K\setminus \delta K,\\
z&\mapsto \left(\frac{1+\delta}{2}\right)h_K(e_1)e_1+z,
\end{split}
\end{equation*}
is well-defined, where $e_1:=(1,0,\dots,0)$. 

To see that the image of $f_{\delta}$ lies inside $K$ notice that, any point $y$ in $f_{\delta}\left(\frac{1-\delta}{2}K\right)$, can be written as
\begin{equation}\label{eq: y}
y=\frac{1-\delta}{2}x+\frac{1+\delta}{2}h_K(e_1)e_1,
\end{equation}
for some $x\in K$. Let $\langle a,x\rangle=b$, for $a\in \mathbb{R}^n$ and $b\in \mathbb{R}$, be a supporting hyperplane for $K$, i.e. $\langle a,x\rangle\leq b$ for every $x\in K$. Then,
\begin{equation*}
\begin{split}
\langle a,y\rangle &=\bigg\langle a,\frac{1-\delta}{2}x+\frac{1+\delta}{2}h_K(e_1)e_1\bigg\rangle,\\
&=\frac{1-\delta}{2}\langle a,x\rangle+\frac{1+\delta}{2}\langle a,h_K(e_1)e_1\rangle,\\
&\leq \frac{1-\delta}{2}b+\frac{1+\delta}{2}b,\\
&\leq b,
\end{split}
\end{equation*}
since $h_K(e_1)e_1\in \partial K$. So, $y\in K$.
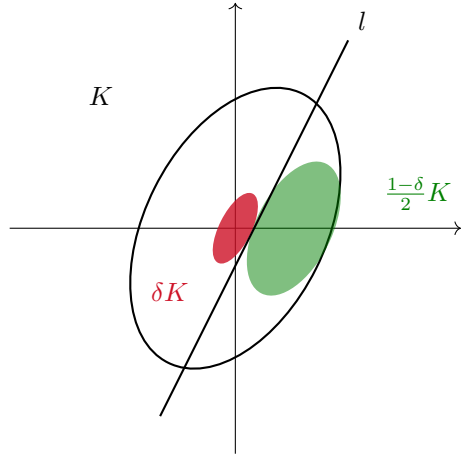
\begin{figure}[htbp] % h: here, t: top, b: bottom, p: page of floats
    \centering % Center the figure
    \begin{minipage}{0.45\textwidth} % Adjust width as needed
		\begin{tikzpicture}

        % Draw axes
        \draw[->] (-3,0) -- (3,0) node[right] {};
        \draw[->] (0,-3) -- (0,3) node[above] {};

        % Helper for rotation: direction (1, 2)
        \def\rotationAngle{63.4349} % arctan(2/1) in degrees

        % Draw larger ellipse (black)
        \draw[thick, rotate around={\rotationAngle:(0,0)}] (0,0) ellipse (2.0cm and 1.2cm);

        % Draw red ellipse (centered at origin)
        \draw[palestinianRed, thick, fill=palestinianRed, opacity=0.8, rotate around={\rotationAngle:(0,0)}] 
          (0,0) ellipse (0.5cm and 0.2cm);
        
        % Draw blue ellipse
        % Adjusted size and position to ensure tangency with x-axis intersections
        \draw[palestinianGreen, thick, fill=palestinianGreen, opacity=0.5, rotate around={\rotationAngle:(0.5,0)}] 
          (0.62,-0.25) ellipse (0.95cm and 0.49cm);

        % Draw diagonal line l
        \draw[thick] (-1,-2.5) -- (1.5,2.5) node[above right] {$l$};

        % Add label to line k
        \node[below left] at (-1.5,2) {$K$};
        \node[below left] at (3,0.8) {\textcolor{palestinianGreen}{$\frac{1-\delta}{2}K$}};
        \node[below left] at (-0.5,-0.6) {\textcolor{palestinianRed}{$\delta K$}};

        \end{tikzpicture}

    \end{minipage}
    \caption{The hyperplane $l$.}
    \label{fig: hyperplane_l} % Add label for referencing
\end{figure}

To see that the image of $f_{\delta}$ lies outside $\delta K$, let $l$ be the supporting hyperplane $\langle a,x\rangle=b$ of $\delta K$, such that $\langle a,\delta h_K(e_1)e_1\rangle=b$. i.e. see Figure \ref{fig: hyperplane_l}. Notice that, similar to \eqref{eq: y}, any point $y$ in $f_{\delta}\left(\frac{1-\delta}{2}K\right)$, can be written as
\begin{equation}\label{eq: y}
y=\frac{1-\delta}{2\delta}x+\frac{1+\delta}{2}h_K(e_1)e_1,
\end{equation}
for some $x\in \delta K$. So,
\begin{equation*}
\begin{split}
\langle a,y\rangle &=\bigg\langle a,\frac{1-\delta}{2\delta}x+\frac{1+\delta}{2}h_K(e_1)e_1\bigg\rangle,\\
&=\frac{1-\delta}{2\delta}\langle a,x\rangle+\frac{1+\delta}{2\delta}b,\\
&\geq -\frac{1-\delta}{2\delta}b+\frac{1+\delta}{2\delta}b,\\
&\geq b,
\end{split}
\end{equation*}
since, given that $K$ is centrally symmetric, we have that $\langle a,x\rangle=-b$ is also a supporting hyperplane, i.e. $\delta K$ is contained in $\langle a,x\rangle\geq -b$.
This concludes the claim on $f_{\delta}$ being well-defined.

This claim implies then, that $\left(\frac{1-\delta}{2}\right)K\times_L K^{\circ}\cong \sqrt{\frac{1-\delta}{2}}(K\times_L K^{\circ})$ symplectically embeds into $K\setminus \delta K\times_L K^{\circ}.$ Therefore, from monotonicity of symplectic capacities, we have that
$$c(K\setminus \delta K \times_L K^{\circ})\geq \frac{1-\delta}{2}c(K\times_L K^{\circ}),$$
as claimed.

\end{proof}

\begin{remark}
In \cite{vicente2025strongviterboconjecturevarious}, the author studies another kind of dual Lagrangian products, arising from $n$-tuples of Young functions and its induced Luxembourg metric. In this case the region $K\subset \mathbb{R}^n$ is also convex and centrally symmetric. Furthermore, a statement as the one in Theorem \ref{thm: Gr_holed_K} also holds. It would be interesting to see, whether for this class of dual Lagrangian products, results analogous to those of Theorems \ref{thm: puncture_higherdim} and \ref{thm: punctured_cube} hold and Lagrangian barriers exist.   
\end{remark} 

\begin{prop}\label{prop: upper_bound_bidisk}
For any normalized symplectic capacity we have that
$$c\left(D^n(1)\setminus D^n(\delta)\times_L D^n(1)\right)\leq 2(1-\delta).$$
\end{prop}
\begin{proof}
Consider the coordinates $(r,\theta,x_3,\ldots,x_n,p_r,p_{\theta},y_3,\ldots,y_n)$ in $T^*\mathbb{R}^n\cong \mathbb{R}^n\times_L \mathbb{R}^n$ obtained by using polar coordinates in the subspace $T^*\mathbb{R}^2$ with coordinates $(x_1,x_2,y_1,y_2)$. Observe that we have the relation 
$$y_1^2+\ldots+y_n^2=p_r^2+\frac{p_{\theta}^2}{r^2}+y_{3}^2+\ldots+y_n^2,$$
and in these coordinates the standard symplectic form reads 
$$dp_r\wedge dr+dp_{\theta}\wedge d\theta+dy_3\wedge dx_3+\ldots+dy_n\wedge dx_n.$$
Similar to the proof of Theorem \ref{thm: NonSqueezing} consider the symplectic embedding $f:D^n(1)\setminus D^n(\delta)\times_L D^n(1)\rightarrow B^2(4)\times \mathbb{R}^{2n-2}$ given by
$$f(r,\theta,x_3,\ldots,x_n,p_r,p_{\theta},y_3,\ldots,y_n)=\left(1+p_{\theta},-\frac{\theta}{1+p_{\theta}},r,p_r,x_3,y_3,\ldots,x_n,y_n\right),$$
whose image is contained in the set
$$W:=\left(B^2(4)\setminus\{(0,0)\}\right)\times \{(x,y)\in \mathbb{R}^2\mid \delta\leq x<1,|y|<1\}\times \mathbb{R}^{2n-4}.$$
For the upper bound, again using area-preserving diffeomorphisms, we can find a symplectic embedding of $W$ into $B^2(4)\times B^2(2(1-\delta)+\varepsilon)\times \mathbb{R}^{2n-4}\subset \mathbb{R}^{2}\times B^2(2(1-\delta)+\varepsilon)\times \mathbb{R}^{2n-4}$, for any $\varepsilon>0$. Composing $f$ with this map, we get a symplectic embedding of $D^n(1)\setminus D^n(\delta)\times_L D^n(1)$ into $B^2(2(1-\delta)+\varepsilon)\times \mathbb{R}^{2n-4}$, for every $\varepsilon>0$ small enough. Using Gromov's Non-Squeezing Theorem we have that 
$$c\left(D^n(1)\setminus D^n(\delta)\times_L D^n(1)\right)\leq2(1-\delta),$$
as claimed.
\end{proof}

%\begin{prop}\label{prop: lower_bound_bidisk}
%$$c_{Gr}\left(B^2(\pi)\setminus\{(0,0)\}\times_L B^2(\pi)\right)\geq2.$$
%\end{prop}
%\begin{proof}
%As for the lower bound, using the same idea as in Proposition \ref{prop: lowerbound}, we can find a symplectic embedding of $\square^2\left(\frac{1-\varepsilon}{2},\frac{1}{2}\right)\times_L\diamond^2(1)$ into $B^2(\pi)\setminus\{(0,0)\}\times_L B^2(\pi)$, for every $\varepsilon>0$ small enough, as follows,
%\begin{equation}\label{eq: def_embedding}
%\begin{split}
%\square^2\left(\frac{1-\varepsilon}{2},\frac{1}{2}\right)\times_L\diamond^2(1)&\rightarrow B^2(\pi)\setminus\{(0,0)\}\times_L B^2(\pi),\\
%(x_1,x_2,y_1,y_2)&\mapsto\left(x_1+\frac{1}{2},x_2,y_1,y_2\right),
%\end{split}
%\end{equation}
%where we are using polar coordinates in the target space. This implies, by results in \cite{Ramos2017OnTR} and \cite{Latschev2011TheGW} once again, that
%$$c_{Gr}\left(B^2(\pi)\setminus\{(0,0)\}\times_L B^2(\pi)\right)\geq2.$$
%This concludes the proof.
%\end{proof}
%
%Now we can put all this together to get the proof of Theorem \ref{thm: puncture_higherdim}.

\begin{proof}[Proof of Theorem \ref{thm: puncture_higherdim}]
The proof of the first statement follows trivially from the statement in Proposition \ref{prop: upper_bound_bidisk}, Theorem  \ref{thm: Gr_holed_K} for $\delta=0$ and Example \ref{ex: D^2xD^2}. For the claim on the minimal action, it's not hard to see that there is indeed\footnote{there is, as a matter of fact, a $S^1$-family of those, invariant by rotation in $\mathbb{R}^2$.} a 2-bouncing billiard trajectory of action $2(1-\delta)$, namely, the one bouncing from the point $(1,0,\dots,0)$ to the point $(1-\delta,0,\dots,0)$. We make the following claim.

\textbf{Claim:} The 2-bouncing trajectory above minimizes action among all the closed orbits. 
\begin{proof}[Proof of Claim]
Let $l_1,\ldots,l_k$ be arcs in the circle of radius 1 representing a closed trajectory of the billiard map in the disk. Let $\alpha_1,\ldots,\alpha_k$ be the angles opposite to the arcs $l_1,\ldots,l_k$ in the triangle formed by these arcs and two radial lines. It is not hard to see that $\textup{length}(l_i)\leq 2\sqrt{1-\delta^2}$ for all $i=1,\ldots,k$, given that $\l_i\subset D^n(1)\setminus D^n(\delta)$. Furthermore, using the cosine law, we can see that $\textup{length}(l_i)=\sqrt{2(1-\cos \alpha_i)}$ for all $i=1,\ldots,k$. We now just need to prove the inequality
$$\sum_{i=1}^{k}\textup{length}(l_i)\geq 2(1-\delta).$$
For this, let $f:(0,\pi)^k\to \mathbb{R}_{\geq 0}$ be given by:
$$f(\alpha_1,\ldots,\alpha_k)=\sum_{i=1}^{k}\sqrt{2(1-\cos \alpha_i)}.$$
We minimize $f$ using Lagrangian multipliers as follows, define $L:(0,\pi)^k\times \mathbb{R}^k\times \mathbb{R}^k\to \mathbb{R}$ as
$$L(\alpha_1,\ldots,\alpha_k,\lambda_1,\dots,\lambda_k,b_1,\dots,b_k)=f(\alpha_1,\ldots,\alpha_k)-\sum_{i=1}^k\lambda_i(\cos \alpha_i-2\delta^2+1-b_i^2).$$
Then 
$$\nabla L=\left(\ldots,\frac{\partial f}{\partial \alpha_i}+\lambda_i\sin \alpha_i,\ldots, \cos \alpha_i-2\delta^2+1-b_i^2,\ldots,2b_i\lambda_i, \ldots\right),$$
and making $\nabla L=0$ we see the following cases;
\begin{itemize}
\item if exists $i$ such that $\lambda_i=0$, then $\frac{\partial f}{\partial \alpha_i}=\frac{\sin \alpha_i}{\sqrt{2(1-\cos\alpha_i)}}=0$ which implies that $\alpha_i=m\pi$ for some interger $m$, and this is not possible.
\item  then for all $i=1,\dots,k$ we must have that $b_i=0$, $\cos \alpha_i=2\delta^2-1$ which implies that $\cos \alpha_j=2\delta^2-1$ for all $j=1,\ldots,k$ because billiard trajectories move along caustics\footnote{A caustic is a closed curve tangent to all the arcs of a given billiard trajectory, see \cite{Tabachnikov2005GeometryAB} for a proof of the fact that circles are caustics for the billiards on the round table.}, hence $b_j=0$ for all $j=1,\ldots,k$. This implies that 
$$\frac{\partial f}{\partial \alpha_i}+\lambda_i\sin \alpha_i=\frac{\sin \alpha_i}{\sqrt{2(1-\cos\alpha_i)}}+\lambda_i\sin \alpha_i=\sin\alpha_i\left(\lambda_i+\frac{1}{\sqrt{2(1-\cos\alpha_i)}}\right)=0,$$ 
and from this we get that, either $\alpha_i=m\pi$ for some integer $m$ or $\cos\alpha_i=1+\frac{1}{2\lambda_i^2}$, the former we already saw is not possible from the first item, while for the latter, substituing into the relation $\cos \alpha_i=2\delta^2-1$ we get that, $\lambda_i=\frac{1}{\sqrt{2\sqrt{\delta^2-1}}}$. From this we get that $(\alpha_1,\dots,\alpha_k)$ is a critical point of $f$ if, and only if, $\cos \alpha_i=2\delta^2-1$ for all $i=1,\dots,k$. For $(\alpha_1,\dots,\alpha_k)$ a critical point of $f$, we have that 
$$f(\alpha_1,\dots,\alpha_k)=2k\sqrt{1-\delta^2}\geq 2(1-\delta),$$
given that $\delta<1$ and $k\geq 2$. Therefore, the claim follows.
\end{itemize}
\end{proof}
This concludes the proof of Theorem \ref{thm: puncture_higherdim}.
\end{proof}

In order to present the proof of Theorem \ref{thm: punctured_cube}. We show first the following Proposition.

\begin{prop}\label{prop: upper_bound_cube}
For any normalized symplectic capacity we have that
$$c\left(\square^n(1,\dots,1)\setminus\{(0,\ldots,0)\}\times_L \diamond^n(1)\right)\leq2.$$
\end{prop}
\begin{proof}
We make the following precise claim: \textit{
For any $0<\lambda<1$, there exists a symplectic embedding $\phi_{\lambda}$ from $\lambda\left(\square^n(1,\dots,1)\times_L \diamond^n(1)\right)$ into $B^{2n}(4)$, satisfying that $\phi_{\lambda}(0,\ldots,0,y_1,\ldots,y_n)\in \mathcal{L}$, for all $(y_1,\ldots,y_n)\in \lambda \diamond^n(1)$, where $\mathcal{L}:=\{(\mathbf{x},\mathbf{y})\in T^*\mathbb{R}^n\mid \mathbf{x}=0, ||\mathbf{y}||\leq 1\}$.} We then have that 
$$c\left(\square^n(1,\dots,1)\setminus\{(0,\ldots,0)\}\times_L \diamond^n(1)\right)\leq \frac{1}{\lambda}c(B^{2n}(4)\setminus \mathcal{L})\leq \frac{2}{\lambda}, \hspace{5mm} \forall \lambda,$$
by direct application of Biran's result, mentioned in Section \ref{sec: intro} and the monotonicity of symplectic capacities. Thus,
$$c\left(\square^n(1,\dots,1)\setminus\{(0,\ldots,0)\}\times_L \diamond^n(1)\right)\leq2.$$

\begin{proof}[Proof of claim]
To prove our claim we follow an idea in \cite{Latschev2011TheGW}, with some modifications. Let $0<\varepsilon<1/\lambda-1$, choose an area preserving embedding 
$$\sigma:B^2(4)\to (-1,1) \times\left(-1-\frac{\varepsilon}{n},1+\frac{\varepsilon}{n}\right),$$
such that:
\begin{enumerate}
\item \label{eq: ineq} $\frac{1}{4}\pi|z|^2\leq |y(\sigma(z))|<\frac{1}{4}\pi|z|^2+\frac{\varepsilon}{n}, \hspace{5mm} \forall z\in B^2(4),$
\item $x(\sigma(z))=0$, whenever $x(z)=0$,
\item \label{it: third} $(\lambda,\lambda)\times \left(-\lambda-\frac{\lambda\varepsilon}{n},\lambda+\frac{\lambda\varepsilon}{n}\right)\subset \textup{Im }\sigma.$
\end{enumerate}

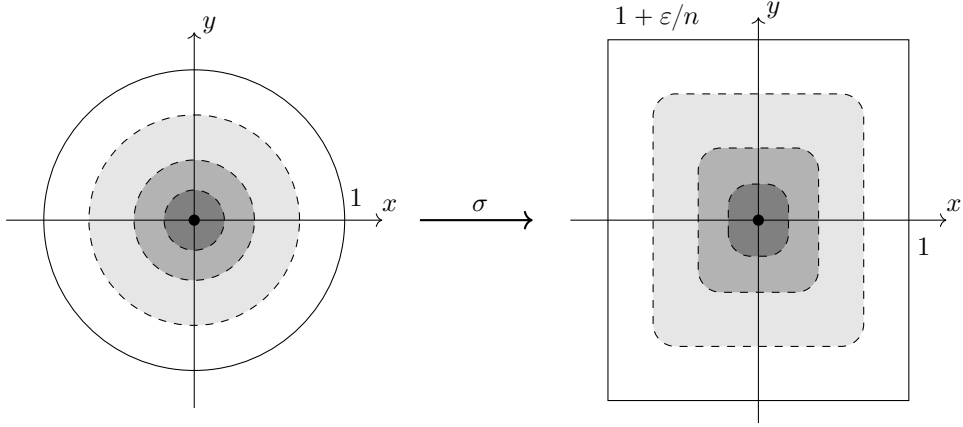
\begin{figure}[htbp] % h: here, t: top, b: bottom, p: page of floats
    \centering % Center the figure
\begin{tikzpicture}

% Left plot (Euclidean norm level sets)
\begin{scope}
\foreach \r/\c [count=\i] in {2.0/none, 1.4/gray!20, 0.8/gray!60, 0.4/gray!100} {
  \ifnum\i=1
    \draw (0,0) circle (\r); % outermost circle: solid, transparent fill
  \else
    \draw[fill=\c, draw=none] (0,0) circle (\r);
    \draw[dashed] (0,0) circle (\r);
  \fi
}
\draw (0,0) node[circle,inner sep=1.5pt,fill=black] {};
\draw[->] (-2.5,0) -- (2.5,0) node at (2.6,0.2) {$x$} node at (2.15,0.3) {$1$};
\draw[->] (0,-2.5) -- (0,2.5) node at (0.2,2.6) {$y$};

\end{scope}

% Arrow
\draw[->,thick] (3,0) -- (4.5,0) node at (3.8,0.2) {$\sigma$};

% Right plot (vertically elongated rectangles, taller y-axis)
\begin{scope}[shift={(7.5,0)}]
\foreach \r/\c [count=\i] in {2.0/none, 1.4/gray!20, 0.8/gray!60, 0.4/gray!100} {
  \pgfmathsetmacro{\yfactor}{1.2}
  \pgfmathsetmacro{\ry}{\r*\yfactor}
  \ifnum\i=1
    \draw (-\r,-\ry) rectangle (\r,\ry); % outermost shape: solid, no fill
  \else
    \draw[fill=\c, draw=none, rounded corners=0.3cm] (-\r,-\ry) rectangle (\r,\ry);
    \draw[dashed, rounded corners=0.3cm] (-\r,-\ry) rectangle (\r,\ry);
  \fi
}
\draw (0,0) node[circle,inner sep=1.5pt,fill=black] {};
% Taller y-axis
\draw[->] (-2.5,0) -- (2.5,0) node at (2.6,0.2) {$x$} node at (2.2,-0.35) {$1$};
\draw[->] (0,-2.7) -- (0,2.7) node at (0.2,2.8) {$y$} node at (-1.35,2.7) {$1+\varepsilon/n$}; % extended y-axis
\end{scope}

\end{tikzpicture}
\caption{The map $\sigma_i$.} % Updated caption
    
\end{figure}

We notice now that the symplectic embedding $\sigma\times\cdots\times \sigma: B^2(4) \times\cdots\times B^2(4) \to \mathbb{R}^{2n}$ maps $B^{2n}(4)$ into $\square^n(1,\ldots,1)\times \diamond^n(1+\varepsilon)$. Indeed, for $(z_1,\ldots,z_n)\in B^{2n}(4)$ we have $\pi(|z_1|^2 + \ldots +|z_n|^2) < 4$. Together with item \eqref{eq: ineq} we can estimate
\begin{equation*}
\begin{split}
\sum_{i=1}^n|y_i\left((\sigma\times\cdots\times \sigma)(z_i,\ldots,z_n)\right)| &=\sum_{i=1}^n|y_i\left(\sigma(z_1)\right)|,\\
&<\frac{\pi}{4}(|z_1|^2+\ldots+|z_n|^2)+\varepsilon,\\
&<1+\varepsilon.
\end{split}
\end{equation*}
Let $\lambda(x_1,\ldots,x_n,y_1,\ldots,y_n)$ with $(x_1,\ldots,x_n,y_1,\ldots,y_n)\in \square^n(1,\dots,1)\times_L \diamond^n(1+\varepsilon)$, we want to show that the image of $\sigma\times\dots\times \sigma$, restricted to $B^{2n}(4)$, contains the point $\lambda(x_1,\ldots,x_n,y_1,\ldots,y_n)$. By item \eqref{it: third}, there exists $(z_1,\ldots,z_n)\in B^2(4)\times \ldots \times B^2(4)$ such that $\sigma(z_i)=\lambda(x_i,y_i)$, for all $i=1,\ldots,n$. Then, item \eqref{eq: ineq} implies that 
\begin{equation*}
\frac{\pi}{4}(|z_1|^2+\ldots+|z_n|^2)\leq \sum_{i=1}^n|y_i\left(\sigma(z_1)\right)|=\lambda \sum_{i=1}^n|y_i|<\lambda(1+\varepsilon)<1,
\end{equation*}
by our choice of $\varepsilon$. Taking $\phi_{\lambda}:=(\sigma\times\cdots\times\sigma)^{-1}|_{\lambda\left(\square^n(1,\dots,1)\times_L \diamond^n(1)\right)}$, the claim follows.
\end{proof}
This concludes the proof.
\end{proof}

We can now finally present the proof of Theorem \ref{thm: punctured_cube}
\begin{proof}[Proof of Theorem \ref{thm: punctured_cube}]
The proof of the first statement follows trivially from Proposition \ref{prop: upper_bound_cube} and Theorem \ref{thm: Gr_holed_K} for $n=2$ and $\delta=0$. For the statement about the minimal action, notice that, using Example \ref{ex: cube_x_diamond}, we have that all the orbits are closed and have the same period equal 4, except exactly 4 orbits (whic are 2-bouncing billiard trajectories), bouncing between the origin and each of the vertices of $\square^2(1,1)$, along a ``half-diagonal". These orbits are easily seen to have action 2.
\end{proof}

\begin{proof}[Proof of Theorem \ref{thm: estimatives}]
For (\ref{item: 1_estimatives}) just notice that 
$$ D^n(\widecheck{k})\times_L D^n\left(\frac{1}{\widehat{k}}\right)\subseteq K\times_L K^{\circ}\subseteq D^n(\widehat{k})\times_L D^n\left(\frac{1}{\widecheck{k}}\right).$$
The claim then follows by a resul in \cite{Ramos2015SymplecticEA} and the monotonicity of symplectic capacities. Similarly, for (\ref{item: 2_estimatives}), we have that 
$$D^n(\widecheck{k})\setminus\{(0,0)\}\times_L D^n\left(\frac{1}{\widehat{k}}\right)\subseteq K\setminus\{(0,0)\}\times_L K^{\circ}\subseteq D^n(\widehat{k})\setminus\{(0,0)\}\times_L D^n\left(\frac{1}{\widecheck{k}}\right).$$
Then, by Theorem \ref{thm: puncture_higherdim}, we have that 
$$2\sqrt{\frac{\widecheck{k}}{\widehat{k}}}\leq c(K\setminus \{(0,0)\}\times_L K^{\circ})\leq 2\sqrt{\frac{\widehat{k}}{\widecheck{k}}},$$
as claimed.\\
For the second part of the claim, notice that 
$$\pi\frac{\widecheck{k}}{\widehat{k}}\leq \sqrt{\textup{Vol}(K)\textup{Vol}(K^{\circ})}\leq \pi\frac{\widehat{k}}{\widecheck{k}}.$$
In light of the pinching condition, we have that $\left(\frac{\widehat{k}}{\widecheck{k}}\right)^3\leq \frac{16}{\pi^2}<\frac{\pi^2}{4}$, so by item (\ref{item: 1_estimatives}) we get the inequality on the right in \eqref{eq: inequality_est}, whereas by item (\ref{item: 2_estimatives}) we get the inequality on the left side of \eqref{eq: inequality_est}.
%Item (\ref{item: 3_estimatives}) follows trivially from items (\ref{item: 1_estimatives}) and (\ref{item: 2_estimatives}).
\end{proof}

We now include the next result in order to show that Lagrangian barriers can also easily be found in other Lagrangian products that are not dual Lagrangian products.

\begin{theorem}\label{thm: disk_times_square}
Let $c$ be any normalized symplectic capacity, then the Lagrangian product $D^2(1)\times_L \square(1,1)$ satisfies that
$$c(D^2(1)\times_L \square(1,1))=4, \hspace{5mm} c(D^2(1)\setminus\{(0,0)\}\times_L \square(1,1))=2.$$
\end{theorem}

\begin{figure}[htbp] % h: here, t: top, b: bottom, p: page of floats
    \centering % Center the figure
\begin{tikzpicture}

% Set parameters
\def\r{1} % You can change this value for A_r

% Draw circle with its own axes
\begin{scope}[shift={(0,0)}]
    % Draw axes for the circle
    \draw[->] (-2,0) -- (2,0) node[right] {$x_1$};
    \draw[->] (0,-2) -- (0,2) node[above] {$x_2$};

    % Draw the circle with radius 1
    \draw[thick] (0,0) circle (1);

    % Mark some points on the circle (optional)
    \filldraw[black] (1.5,0) circle (0pt) node[anchor=south] {$(1,0)$};

    \node at (1.2, 1.5) {\large $D^2(1)$};
\end{scope}

% Draw square with its own axes and increased spacing
\begin{scope}[shift={(6,0)}] % Adjusted horizontal shift to 6 (for 30mm)
    % Draw axes for the square
    \draw[->] (-2,0) -- (2,0) node[right] {$y_1$}; 
    \draw[->] (0,-2) -- (0,2) node[above] {$y_2$}; 

    % Draw the square centered at the origin with sides of length 2
    \draw[thick] 
        (1,1) -- (1,-1) -- (-1,-1) -- (-1,1) -- cycle;

    % Draw vertices of the square
    \filldraw[black] (1,1) circle (0pt) node[anchor=south west] {$(1,1)$};
    \filldraw[black] (1,-1) circle (0pt) node[anchor=north west] {$(1,-1)$};
    \filldraw[black] (-1,-1) circle (0pt) node[anchor=north east] {$(-1,-1)$};
    \filldraw[black] (-1,1) circle (0pt) node[anchor=south east] {$(-1,1)$};

    \node at (1.2, 2) {\large $\square^2(2)$};
\end{scope}

\end{tikzpicture}
\caption{The Lagrangian product $B^2(\pi)\times_L \square^2(2)$.} % Updated caption
\end{figure}
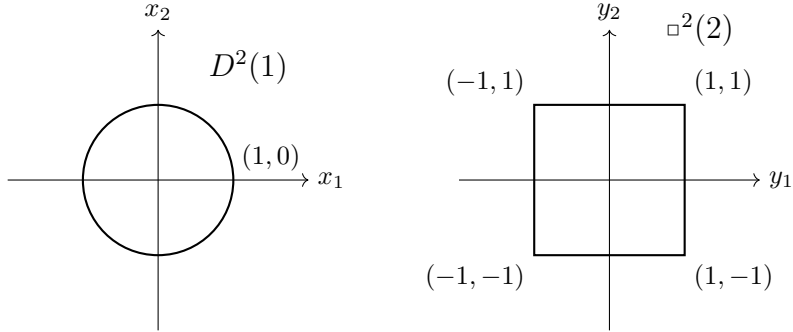

\begin{proof}
Notice that $\diamond^2(1)\times_L\square(1,1)$ is naturally sitting inside $D^2(1)\times_L\square(1,1)$. So, by results in \cite{Ramos2017OnTR}, we have that $c(D^2(1)\times_L\square(1,1))\geq 4$. On the other hand, the existence of the symplectic embedding $f:D^2(1)\times_L\square(1,1)\rightarrow \square(1,1)\times \mathbb{R}^2$ given by $f(x_1,x_2,y_1,y_2)=(x_1,y_1,x_2,y_2)$ and Gromov's Non-Squeezing Theorem imply that $c(D^2(1)\times_L\square(1,1))\leq 4$. This proves the first claim of the theorem.

For the second part, define $\triangle(r):\{(x,y)\in \mathbb{R}^2_{\geq 0}\mid x+y<r\}$. Then, observe that $\triangle(1-\varepsilon)\times_L\square(1,1)$ is naturally sitting inside $D^2(1)\setminus{(0,0)}\times_L\square(1,1)$, for every $\varepsilon>0$ sufficiently small. This implies, by Traynor's trick in \cite{Traynor1995SymplecticPC}, that $c(D^2(1)\setminus {(0,0)}\times_L\square(1,1))\geq 2$. For the upper bound, we can use exactly the same symplectic embedding as in the proof of Proposition \ref{prop: upper_bound_bidisk}, to get an embeeding into $B^2(2+\varepsilon)\times{\mathbb{R}^2}$ for every $\varepsilon>0$ and then by Gromov's Non-Squeezing Theorem we get that $c(D^2(1)\setminus {(0,0)}\times_L\square(1,1))\leq 2$.

\end{proof}

\begin{remark}
    It is easy to see that not every Lagrangian plane of the form $\{\mathbf{x}_0\}\times \mathbb{R}^n$, for $\mathbf{x}_0\in \mathbb{R}^n$, is a Lagrangian barrier. Indeed, for any $r>0$, let $A_{r}\subset\mathbb{R}^2$ be defined as the convex hull of the set of points 
$$\{(0,1),(2,1),(2,-1),(0,-1),(-r,0)\}.$$ 
For the Lagrangian product $A_r\times_L \diamond^2(1)$, we have that 
$$c(A_r\times_L \diamond^2(1))=c(A_r\setminus\{(0,0)\}\times_L \diamond^2(1))=4,$$
for every $r>0$. 

To see this, notice that $\square^2(1-\varepsilon,1-\varepsilon)\times_L\diamond^2(1)$ symplectically embeds into both $A_r\times_L\diamond^2(1)$ and $A_r\setminus \{(0,0)\}\times_L\diamond^2(1)$, for every $\varepsilon>0$ sufficiently small. Such embedding is clearly decribed in Figure \ref{fig: lag_product_embedding}. 

\begin{figure}[htbp] % h: here, t: top, b: bottom, p: page of floats
    \centering % Center the figure
\begin{tikzpicture}

% Set parameters
\def\r{1} % Radius parameter for A_r
\def\eps{0.2} % Small epsilon value

% Draw A_r with its own axes
\begin{scope}[shift={(0,0)}]
    % Draw axes for A_r
    \draw[->] (-2,0) -- (3,0) node[right] {$x_1$};
    \draw[->] (0,-2) -- (0,2) node[above] {$x_2$};

    % Draw the convex hull A_r (no filling)
    \draw[thick] 
        (0,1) -- (2,1) -- (2,-1) -- (0,-1) -- (-\r,0) -- cycle;

    % Draw vertices of A_r
    \filldraw[black] (-0.5,1) circle (0pt) node[anchor=south] {$(0,1)$};
    \filldraw[black] (2,1) circle (0pt) node[anchor=south] {$(2,1)$};
    \filldraw[black] (2,-1) circle (0pt) node[anchor=north] {$(2,-1)$};
    \filldraw[black] (-0.6,-1) circle (0pt) node[anchor=north] {$(0,-1)$};
    \filldraw[black] (-\r,0.4) circle (0pt) node[anchor=east] {$(-r,0)$};
    \filldraw[black] (0,0) circle (3pt) node[anchor=west] {};
    \node at (1, 1.5) {\large $A_r$};

    % Draw and fill the square centered at (1,0) with side length 2-\varepsilon
    \fill[RoyalBlue!70] % Light Israeli flag blue fill
        (\eps,1-\eps) -- (2-\eps,1-\eps) -- (2-\eps,-1+\eps) -- (\eps,-1+\eps) -- cycle;
    \draw[thick, dashed] % Dashed outline
        (\eps,1-\eps) -- (2-\eps,1-\eps) -- (2-\eps,-1+\eps) -- (\eps,-1+\eps) -- cycle;

\end{scope}

% Draw diamond with its own axes and increased spacing
\begin{scope}[shift={(6,0)}] % Adjusted horizontal shift to 6 (for 30mm)
    % Draw axes for diamond^2(1)
    \draw[->] (-2,0) -- (2,0) node[right] {$y_1$};
    \draw[->] (0,-2) -- (0,2) node[above] {$y_2$};

    % Draw and fill the diamond shape diamond^2(1)
    \fill[RoyalBlue!70] % Light Israeli flag blue fill
        (1,0) -- (0,1) -- (-1,0) -- (0,-1) -- cycle;
    \draw[thick] % Outline
        (1,0) -- (0,1) -- (-1,0) -- (0,-1) -- cycle;

    % Draw vertices of diamond^2(1)
    \filldraw[black] (2,0.4) circle (0pt) node[anchor=east] {$(1,0)$};
    \filldraw[black] (-0.6,1) circle (0pt) node[anchor=south] {$(0,1)$};
    \filldraw[black] (-2,-0.4) circle (0pt) node[anchor=west] {$(-1,0)$};
    \filldraw[black] (0.6,-1) circle (0pt) node[anchor=north] {$(0,-1)$};
    \node at (1, 1.5) {\large $\diamond^2(1)$};
\end{scope}

\end{tikzpicture}
\caption{Embedding of $\square^2(1-\varepsilon,1-\varepsilon)\times_L\diamond^2(1)$ in the Lagrangian product $A_r\times_L \diamond^2(1)$.} % Updated caption
    \label{fig: lag_product_embedding} % Add label for referencing
\end{figure}
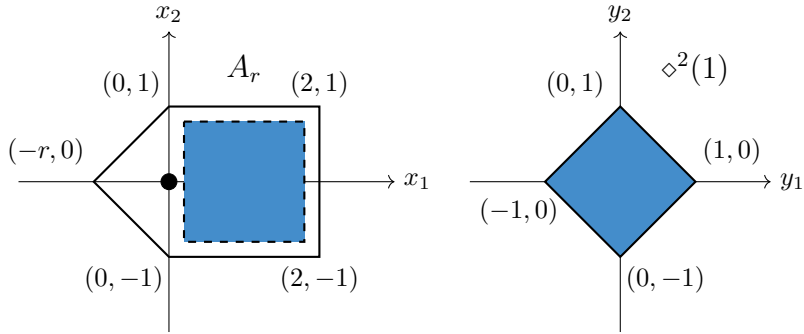

This implies $c(A_r\times_L\diamond^2(1))\geq 4$ and $c(A_r\setminus \{(0,0)\}\times_L\diamond^2(1))\geq 4$. Observe now that the map $f:A_r\times_L\diamond^2(1)\rightarrow \square^2(1,1)\times \mathbb{R}^2$ given by $f(x_1,x_2,y_1,y_2):=(x_2,y_2,x_1,y_2)$ is a symplectic embedding. So, by Gromov's Non-squeezing Theorem we have that $c(A_r\times_L\diamond^2(1))\leq 4$ and hence, by monotonicity of symplectic capacities, $c(A_r\setminus \{(0,0)\}\times_L\diamond^2(1))\leq 4$. 
\end{remark}

\begin{figure}[htbp] % h: here, t: top, b: bottom, p: page of floats
    \centering % Center the figure
\begin{tikzpicture}

% Set parameters
\def\r{1} % You can change this value for A_r

% Draw A_r with its own axes
\begin{scope}[shift={(0,0)}]
    % Draw axes for A_r
    \draw[->] (-2,0) -- (3,0) node[right] {$x_1$};
    \draw[->] (0,-2) -- (0,2) node[above] {$x_2$};

    % Draw the convex hull A_r (no filling)
    \draw[thick] 
        (0,1) -- (2,1) -- (2,-1) -- (0,-1) -- (-\r,0) -- cycle;

    % Draw vertices of A_r
    \filldraw[black] (-0.5,1) circle (0pt) node[anchor=south] {$(0,1)$};
    \filldraw[black] (2,1) circle (0pt) node[anchor=south] {$(2,1)$};
    \filldraw[black] (2,-1) circle (0pt) node[anchor=north] {$(2,-1)$};
    \filldraw[black] (-0.6,-1) circle (0pt) node[anchor=north] {$(0,-1)$};
    \filldraw[black] (-\r,0.4) circle (0pt) node[anchor=east] {$(-r,0)$};

    \node at (1, 1.5) {\large $A_r$};
\end{scope}

% Draw diamond with its own axes and increased spacing
\begin{scope}[shift={(6,0)}] % Adjusted horizontal shift to 6 (for 30mm)
    % Draw axes for diamond^2(1)
    \draw[->] (-2,0) -- (2,0) node[right] {$y_1$}; % Changed x_1 to y_1
    \draw[->] (0,-2) -- (0,2) node[above] {$y_2$}; % Changed x_2 to y_2

    % Draw the diamond shape diamond^2(1) (no filling)
    \draw[thick] 
        (1,0) -- (0,1) -- (-1,0) -- (0,-1) -- cycle;

    % Draw vertices of diamond^2(1)
    \filldraw[black] (2,0.4) circle (0pt) node[anchor=east] {$(1,0)$};
    \filldraw[black] (-0.6,1) circle (0pt) node[anchor=south] {$(0,1)$};
    \filldraw[black] (-2,-0.4) circle (0pt) node[anchor=west] {$(-1,0)$};
    \filldraw[black] (0.6,-1) circle (0pt) node[anchor=north] {$(0,-1)$};

    \node at (1, 1.5) {\large $\diamond^2(1)$};
\end{scope}

\end{tikzpicture}
\caption{The Lagrangian product $A_r\times_L \diamond^2(1)$.} % Add caption
    \label{fig: lag_product} % Add label for referencing
\end{figure}
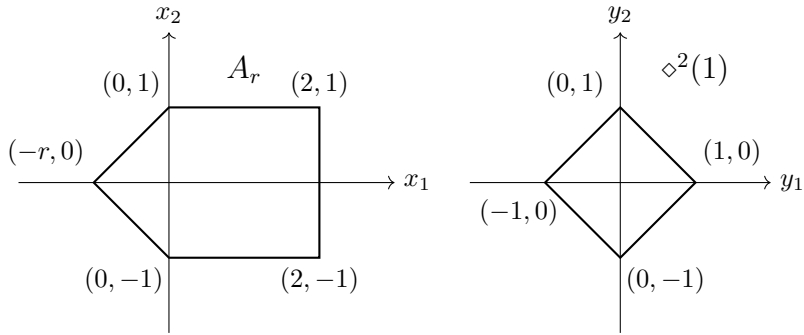

%However, we will give a second proof of this upper bound using ideas similar to \cite{Biran2001LagrangianBA}, in the hope that this should shed some light on how to prove these kind of statements in a general case. Consider the cylinder $E$ inside $B^2(\pi)\times_L\square(1,1)$, defined in polar coordinates by,
%$$E:=\{(r,\theta,p_r,p_{\theta})\mid r=1/2,p_r=0\}.$$
%Now, consider the symplectic fibration $\pi:B^2(\pi)\times_L\square(1,1)\rightarrow E$ defined by $\pi(r,\theta,p_r,p_{\theta}):=(\theta,p_{\theta}).$
%For every $(\theta,p_{\theta})\in E$, we have that $\pi^{-1}(\theta,p_{\theta})$ is a disk of area 2. Using the theory of pseudoholomorphic curves as in \cite{Gromov1985PseudoHC} and \cite{Biran2001LagrangianBA}, we get the necessary obstruction for the upper bound.
\

\section{The disk cotangent bundle of cylinders}\label{sec: disk_cot}

Let $g:\mathbb{R}\rightarrow \mathbb{R}_{\geq 0}$ given by 
\begin{equation*}
  g(a)= 
\begin{cases}
    \begin{split}
       4a&, \textup{ if } a\leq \pi\\
        4\pi&, \textup{ if } a\geq \pi.
    \end{split}
    \end{cases}
\end{equation*}

\begin{prop}\label{prop: upperbound}
For any normalized symplectic capacity $c$ we have that
$$c(D^*C_a)\leq g(a).$$ 
\end{prop}
\begin{proof}
As in the proof of Theorem \ref{thm: NonSqueezing} consider the symplectic embedding $f: D^*C_a\rightarrow B^2\left(4\pi\right)\times \mathbb{R}^2$ given by
    $$f(z,\theta,p_z,p_{\theta})=\left(1+p_{\theta},-\frac{1}{1+p_{\theta}}\theta,z,p_z\right),$$
    whose image is contained in the subset 
    $$W:=\{(r,\theta, x,y)\in \mathbb{R}^4\mid 0<r<2,\theta\in [0,2\pi),-a\leq x\leq a, -1\leq y\leq 1 \}\subset B^2\left(4\pi\right)\times \mathbb{R}^2.$$
We focus first on the upper bound. Using area preserving diffeomorphisms, it can be shown that, for any $\varepsilon>0$ there is a symplectic embeddding from $W$ into 
$B^2\left(4\pi\right)\times B^2(4a+\varepsilon)$. Notice now that $B^2\left(4\pi\right)\times B^2(4a+\varepsilon)$ is naturally contained in both $Z^4(4\pi)$ and $Z^4(4a+\varepsilon)$. This implies that, for $a\leq \pi$, $W$ symplectically embeds into $B^2(4a+\varepsilon)\times \mathbb{R}^2$ and for $a>\pi$, $W$ symplectically embeds into $B^2(4\pi)\times \mathbb{R}^2$. Using Gromov's Non-Squeezing Theorem we can see that 
$$c(D^*C_a)\leq g(a).$$        
\end{proof}

\begin{prop}\label{prop: lowerbound}
There exists a symplectic embedding of $B^4(g(a)-2\varepsilon)$ into $D^*C_a$, for every $\varepsilon>0$.
\end{prop}
\begin{proof}
 It is not hard to see there exists a symplectic embedding $h:\square\left(a-\frac{\varepsilon}{2}, \pi-\frac{\varepsilon}{2}\right)\times_L \diamond(1)\rightarrow D^*C_a$, for every $\varepsilon>0$, given by
$$h(x_1,x_2,y_1,y_2):=(x_1,x_2,y_1,y_2),$$
where we took cartesian coordinates in $\square\left(a-\frac{\varepsilon}{2}, \pi-\frac{\varepsilon}{2}\right)\times_L \diamond(1)$ and cylindrical coordinates in the target $D^*C_a$. By results in \cite{Ramos2017OnTR} and \cite{Latschev2011TheGW}, we have there is a symplectic embedding of $B^4\left(2\min(2a-\varepsilon,2\pi-\varepsilon)\right)$ into $D^*C_a$, for every $\varepsilon>0$. Hence, the claim follows.

\end{proof}

\begin{proof}[Proof of Theorem \ref{thm: Gromov_width}]
It follows by putting together Propositions \ref{prop: upperbound} and \ref{prop: lowerbound}.  
\end{proof}

\bibliographystyle{unsrt}
\bibliography{Biblio}

\end{document}